\documentclass[12pt]{amsart}

\usepackage{lscape,amssymb, amsmath,verbatim,graphicx}
\usepackage{epsfig,float}
\usepackage{graphicx}
\usepackage{epsfig}
\newcommand{\lf}{\lfloor}
\newcommand{\rf}{\rfloor}
\newcommand{\beps}{{\mbox{\boldmath $\epsilon$}}}

\newcommand{\bet}{{\mbox{\boldmath $\eta$}}}
\newcommand{\bX}{{\bf X}}
\newcommand{\one}{\chi^{(1)}}
\newcommand{\two}{\chi^{(2)}}
\newcommand{\sa}{\begin{align*}
}
\newcommand{\se}{\end{align*}
}
\numberwithin{table}{section}
\numberwithin{figure}{section}
\newcommand{\bY}{{\bf Y}}

\newcommand{\intt}{\int\hspace{-.2cm}\int}

\newtheorem{theorem}{Theorem}[section]
\newtheorem{lemma}{Lemma}[section]
\newtheorem{assumption}{Assumption}[section]

\newcommand\Y{{\bf Y}}
\newcommand\X{{\bf X}}

\theoremstyle{definition}
\newtheorem{defi}{Definition}[section]

\allowdisplaybreaks
\def\beq{\begin{equation}}
\def\eeq{\end{equation}}

\numberwithin{equation}{section}
\numberwithin{theorem}{section}

\begin{document}
\title[Testing for independence]{ Testing for independence between   functional time series}

\author {Lajos Horv\'ath}
%%$^{(1)}$}

\address{Lajos Horv\'ath, Department of Mathematics, University of Utah, Salt Lake City, UT 84112--0090 USA
%\\ email: horvath@math.utah.edu
}

\author{Gregory Rice}
\address{Gregory Rice, Department of Mathematics, University of Utah, Salt Lake City, UT 84112--0090 USA
}

%\subjclass{60F17, 60F25, 62G10}


\thanks{
Research supported by NSF grant DMS  1305858\\
}

\begin{abstract}
Frequently econometricians are interested in verifying a
relationship between two or more time series. Such analysis is typically
carried out by  causality and/or independence tests which have been well
studied when the data is univariate or multivariate. Modern data though is
increasingly of a high dimensional or functional nature for which finite dimensional methods are  not suitable. In
the present paper we develop methodology to check the assumption that data
obtained from two functional time series are independent. Our procedure is
based on the norms of empirical cross covariance operators and is
asymptotically validated when the underlying populations are assumed to be
in a class of weakly dependent random functions which include the
functional ARMA, ARCH and GARCH processes.

\end{abstract}

\maketitle

\section{Introduction and results}\label{first}
A common goal of data analysis in econometrics is to determine whether or not a relationship exists  between two variables which
are measured over time. A determination in either way may be useful. On one hand, if a relationship is confirmed to exist between two variables
then further investigation into the strength and nature of the
relationship may lead to interesting insights or effective predictive models.  Conversely if there is no
relationship between the two variables then an entire toolbox of
statistical techniques developed to analyze two samples which are
independent may be used. The problem of testing for correlation between
two univariate or multivariate time series has been well treated, and we
discuss the relevant literature below. However, as a by product of
seemingly insatiable modern data storage technology, many data of interest
exhibit such large dimension that traditional multivariate techniques are not
suitable. For example, tick by tick stock return data is stored hundreds
of times per second, leading to thousands of observations during a single
day. In such cases a pragmatic approach is to treat the data as densely
observed measurements from an underlying curve, and, after using the
measurements to approximate the curve, apply statistical techniques to the
curves themselves. This approach is fundamental in functional data
analysis, and in recent years much effort has been put forth to adapt
currently available procedures in multivariate analysis to functional
data. The goal of the present paper is to develop a test for independence
between two functional time series.\\
In the context of checking for independence between two second order
stationary univariate time series, Haugh (1976)  proposed a testing
procedure based on sample cross--correlation estimators. His test may be
considered as an adaptation of the popular Box--Ljung--Pierce portmanteau
test (cf.\ Ljung and Box (1978)) to two samples.  In a similar progression
the multivariate portmanteau test of Li and Mcleod (1981) was extended
to test for correlation between two multivariate ARMA time series by El Himdi
and Roy (1997) whose test statistic was based on cross--correlation
matrices. The literature on such tests has also grown over the years to
include adaptations for robustness as well as several other
considerations, see Koch and Yang (1986), Li and Hui (1996) and El Himdi et
al (2003) for details. Many of these results are summarized in Li (2004).
A separate approach for multivariate data based on the  distance correlation measure is developed  in Sz\'ekely and Rizzo (2013).\\
The analysis of functional time series has seen increased attention in
statistics, economics  and in other applications over the last decade, see Horv\'ath and Kokoszka
(2012) for a summary of recent advances. To test for independence within a
single functional time series, Gabrys and Kokoszka (2007) proposed a
method where the functional observations are projected onto finitely many
basis elements, and a multivariate portmanteau test is applied to the
vectors of scores which define the projections. Horv\'ath et al (2013)
developed a portmanteau test for functional data in which the inference is
performed using the operator norm of the empirical covariance operators at
lags $h$, $1 \le h \le H$, which could be considered as a direct
functional analog of the Box--Ljung--Pierce test. Due to the infinite
dimension of functional data, a normal limit is established for the test
statistic rather than the classical $\chi^2$ limit with degrees of freedom
depending on the data dimension.\\
The method that we propose for testing noncorrelation between two
functional time series follows the example of Horv\'ath et al (2013).
Suppose that we have observed $X_1(t),\ldots,X_n(t)$ and
$Y_1(s),\ldots,Y_n(s)$ which are samples from jointly stationary sequences
of square integrable random functions on [0,1]. Formally we are interested
in testing
$$
H_0:\mbox{the sequences}\;\;
\{X_i\}_{i=1}^\infty\;\;\mbox{and}\;\;\{Y_j\}_{j=1}^\infty\;\;\mbox{are
independent}
$$
against the alternative
\begin{align*}
H_A:&\mbox{ for some integer $h_0$, $-\infty < h_0 < \infty$,  $\intt
C_{h_0}^2(t,s)dtds>0$} \\
&\mbox{ where $C_{h_0}(t,s)=\mbox{Cov}(X_0(t),Y_{h_0}(s))$.}
\end{align*}
We use the notation $\int$ to mean $\int_0^1$. Assuming jointly Gaussian distributions for the underlying observations, independence reduces to zero cross--correlations at all lags, and hence  $H_A$ is equivalent to the complement of $H_0$ in that case.
To derive the
test statistic, we note that under $H_0$, the sample cross--covariance functions
$$
\hat{C}_{n,h}(t,s)=\begin{cases}
\displaystyle \frac{1}{n}\sum_{i=1}^{n-h} (X_i(t)-\bar{X}(t))(Y_{i+h}(s)-\bar{Y}(s)) &
h\ge 0 \vspace{.3 cm}\\
\displaystyle\frac{1}{n}\sum_{i=1-h}^{n} (X_i(t)-\bar{X}(t))(Y_{i+h}(s)-\bar{Y}(s)) & h<0
\end{cases}
$$
should be close to the zero function for all choices of $h$, where
$$
\bar{X}(t)=\frac{1}{n}\sum_{i=1}^n
X_i(t),\;\;\mbox{and}\;\;\bar{Y}(s)=\frac{1}{n}\sum_{i=1}^n Y_i(s).
$$
Under $H_A$ a cross covariance function is different
from the zero function for at least one $h$. The test statistic is then
based on
$$
\hat{T}_{n,H}= \sum_{h=-H}^H \intt \hat{C}^2_{n,h}(t,s)dtds
$$
with the hope that it includes the covariance estimator corresponding to
a non zero function if it exists. We then reject $H_0$ for large values of
$\hat{T}_{n,H}$. Our main result is the asymptotic distribution of $\hat{T}_{n,H}$ under $H_0$.\\
In order to cover a large class of functional time series processes, we assume that $\X=\{X_i\}_{i=-\infty}^\infty$ and $\bY=\{Y_i\}_{i=-\infty}^\infty$ are approximable Bernoulli shifts.
We say that $\bet=\{\eta_{j}(t)\}_{j=-\infty}^\infty$ is an $L^4$ absolutely
approximable Bernoulli shift in  $\{\epsilon_{j}(t), -\infty<j<\infty\}$ if
\begin{align}\label{ber-1}
&\eta_{i}=g(\epsilon_{i},\epsilon_{i-1},...)\;\;
\mbox{for some nonrandom measurable function}\\
& g:S^\infty \mapsto
L^2 \mbox{  and i.i.d.\ random innovations}\; \epsilon_{j},\;-\infty<j<\infty,\notag\\
&\mbox{with values in a  measurable space}\;\; S,\notag
\end{align}
\begin{align}\label{ber-2}
\eta_{j}(t)=\eta_{j}(t,\omega)\;\;\mbox{is jointly measurable
in}\;\;(t,\omega)\;\;(-\infty <j<\infty ),
\end{align}
and
\begin{align}\label{ber-3}
 &\mbox{the sequence} \;\; \{\bet\} \;\;\mbox{can be approximated by} \;\;
\ell\mbox{--dependent sequences}\\
 &\{\eta_{j,\ell}\}_{j=-\infty}^\infty \;\mbox{in the sense that}
\;\;\notag\\
& \hspace{3 cm}
\sum_{\ell=1}^\infty \ell(E\|\eta_{j}-\eta_{j,\ell}\|^{4})^{1/4}<\infty \notag\\
&\mbox{where}\;\;\; \eta_{j,\ell}\;\;\mbox{is defined by}\;\;
\eta_{j,\ell}=g(\epsilon_{j},\epsilon_{j-1},...,\epsilon_{j-\ell+1},\beps_{j,\ell}^*),
\notag\\
&\beps^*_{j,\ell}=(\epsilon_{j,\ell, j-\ell}^*,\epsilon_{j,\ell,
j-\ell-1}^*,\ldots), \mbox{where the }\;\; \epsilon^*_{j,\ell,k}\mbox{'s}
\;\;\mbox{are independent copies of}%\;\; \epsilon_{0},
\notag\\
&\epsilon_{0},\;\mbox{independent of} \;\{\epsilon_{j},
-\infty<j<\infty\}. \notag
\end{align}
In assumption \eqref{ber-1} we take $S$ to be an arbitrary measurable space, however in most applications $S$ is itself a function space and the evaluation of $ g(\epsilon_{i},\epsilon_{i-1},...)$ is a functional of $\{\epsilon_j(t)\}_{j=-\infty}^i$. In this case assumption \eqref{ber-2} follows from the joint measurability of the  $\epsilon_i(t,\omega)$'s.
Assumption  \eqref{ber-3} is stronger than the requirement $\sum_{\ell=1}^\infty (E\|\eta_{j}-\eta_{j,\ell}\|^{2})^{1/2}<\infty$ used by H\"ormann and Kokoszka (2010), Berkes et al (2013) and Jirak (2013) to establish the central limit theorem for sums of Bernoulli shifts. Since we need the central limit theorem for sample correlations, higher moment conditions and a faster rate in the approximability  with $\ell$--dependent sequences are needed.\\
We assume that the  sequences $\X$ and $\Y$ satisfy the following  conditions:
\begin{assumption}\label{as-2}\;\; $E\|X_0\|^{4+\delta}<\infty$\;and \;$E\|Y_0\|^{4+\delta}<\infty$\;\;\;with some \;$\delta>0$,
\end{assumption}
\begin{assumption}\label{as-3}\;\; $\bX=\{X_{i}(t)\}_{i=-\infty}^\infty$ is an $L^4$
absolutely approximable Bernoulli shift in $\{\epsilon_{j}(t), -\infty<j<\infty\}$,
\end{assumption}
and
\begin{assumption}\label{as-4}\;\; $\bY=\{Y_{i}(t)\}_{i=-\infty}^\infty$ is an $L^4$ absolutely
approximable Bernoulli shift in  $\{\bar{\epsilon}_{j}(t), -\infty<j<\infty\}$.
\end{assumption}
The functions defining the Bernoulli shift sequences $\X$ and $\Y$ as in \eqref{ber-1} will be denoted by $g_\X$ and $g_\Y$, respectively. The independence of the sequences $\bX$ and $\bY$ stated under $H_0$ is conveniently given by:
\begin{assumption}\label{as-5}\;\; The sequences $\{\epsilon_{j}(t), -\infty<j<\infty\}$ and $\{\bar{\epsilon}_{j}(t), -\infty<j<\infty\}$ are independent.
\end{assumption}
The parameter $H$ defines the number of lags used to
define the test statistic. As $n$ increases it becomes possible to
accurately estimate cross covariances for larger lags, and thus we allow
$H$ to tend to infinity with the sample size. Namely,
\begin{assumption}
\label{as-1} \;\;$H=H(n)\to \infty$ and  ${H}{n^{-\tau}}\to 0$, as  $n\to \infty$, with some $0<\tau< 2\delta/(4+7\delta)$, where $\delta$ is defined in Assumption \ref{as-2}.
\end{assumption}

  To state the limit  result for $\hat{T}_{n,H}$ we first introduce the asymptotic expected value and variance.  Let for all $-\infty<j<\infty$
$$
\gamma_{\X}(j)=\int \mbox{cov}(X_0(t), X_j(t))dt,
\quad
\gamma_{\Y}(j)=\int \mbox{cov}(Y_0(t), Y_j(t))dt
$$
and define
\beq\label{a-def}
\mu=%\gamma_\X(0)\gamma_\Y(0)+2
\sum_{j=-\infty}^\infty\gamma_{\X}(j)\gamma_{\Y}(j).
\eeq
It is shown in Lemma \ref{l-var}  that under Assumptions \ref{as-3} and \ref{as-4} the infinite sum in the definition of $\mu$ above is absolutely convergent. Let
$$
\sigma^2_h=2\int \!\!\!\cdots\!\!\!\int\left(\sum_{\ell=-\infty}^\infty\mbox{cov}(X_0(t), X_\ell(s))\mbox{cov}(Y_0(u), Y_{\ell+h}(v))\right)^2dtdsdudv
$$
and
\beq\label{sigdef}
\sigma^2=\sum_{h=-\infty}^\infty \sigma_h^2.
\eeq
\begin{theorem}\label{main} If Assumptions \ref{as-2}--\ref{as-1} hold, then we have
\begin{equation*}
\frac{n\hat{T}_{n,H}-(2H+1)\mu}{(2H+1)^{1/2}\sigma}\;\;\;\stackrel{{\mathcal D}}{\to}\;\;\;{\mathcal N},
\end{equation*}
where ${\mathcal N}$ stands for a standard normal random variable.
\end{theorem}

Upon the estimation of $\mu$ and $\sigma^2$ this result supplies an asymptotic test for $H_0$.\\
The rest of the paper is organized as follows. In Section \ref{applic} we discuss the estimation of the parameters appearing in  Theorem \ref{main} as well as a simulation study of how the limit result is manifested in  finite sample sizes. The statistical utility of Theorem \ref{main} is then  illustrated by an application to tick data from several stocks listed on the NYSE. The proof of the main result is given in Section \ref{proofs}. The paper concludes with two appendices which contain some technical  results  needed in Section \ref{proofs}.

\section{Finite sample properties and an application} \label{applic}\setcounter{equation}{0}
\subsection{Parameter estimates}\label{parem}
In order to use Theorem \ref{main} for data analysis it is necessary  to estimate $\mu$ and $\sigma^2$ from the sample. Let
$$
\hat{\gamma}_{\X,\ell}(t,s)=\begin{cases}
\displaystyle \frac{1}{n}\sum_{i=1}^{n-\ell} (X_i(t)-\bar{X}(t))(X_{i+\ell}(s)-\bar{X}(s)) &
\ell\ge 0 \vspace{.3 cm}\\
\displaystyle\frac{1}{n}\sum_{i=1-\ell}^{n} (X_i(t)-\bar{X}(t))(X_{i+\ell}(s)-\bar{X}(s)) & \ell<0
\end{cases}
$$
and
$$
\hat{\gamma}_{\Y,\ell}(t,s)=\begin{cases}
\displaystyle \frac{1}{n}\sum_{i=1}^{n-\ell} (Y_i(t)-\bar{Y}(t))(Y_{i+\ell}(s)-\bar{Y}(s)) &
\ell\ge 0 \vspace{.3 cm}\\
\displaystyle\frac{1}{n}\sum_{i=1-\ell}^{n} (Y_i(t)-\bar{Y}(t))(Y_{i+\ell}(s)-\bar{Y}(s)) & \ell<0.
\end{cases}
$$
Since $\mu$ is an infinite sum of integrated correlations, it is natural to use a kernel estimator   of the form
$$
\hat{\mu}_n=\sum_{\ell=1-n}^{n-1}K_1\left(\frac{\ell}{w_1}\right)\intt \hat{\gamma}_{\X,\ell}(t,s)dtds\intt\hat{\gamma}_{\Y,\ell}(t,s)dtds,
$$
where $K_1$ is a kernel function with window  $w_1=w_1(n)$. According to the definition in \eqref{sigdef}, $\sigma^2$ is an infinite sum of integrals of squared  correlations and  so we can use again a kernel type estimator:
$$
\hat{\sigma}_n^2=\sum_{h=-2H}^{2H}K_2\left(\frac{h}{w_2}\right)\hat{\tau}_{n,h},
$$
where $K_2$ is a kernel with window  $w_2=w_2(H)$,
$$
\hat{\tau}_{n,h}=2\int\!\!\!\cdots\!\!\!\int\left(\sum_{\ell=h-n}^{n-h} \left(1-\frac{|\ell|}{n}\right)\hat{\gamma}_{\X,\ell}(t,s)\hat{\gamma}_{\Y, \ell+h}(u,v)\right)^2dtdsdudv.
$$
We  note that $\sigma^2$ is essentially the long run variance of the integrals $\intt \hat{C}_{n,h}^2(t,s)dtds, -2H\leq h \leq 2H$, and $\tau_h$ is the correlation of lag $h$ between these  integrals. Hence $\hat{\tau}_{n,h}$ can be considered as an estimated correlation and $\hat{\sigma}_n^2$ is a kernel  long run variance estimator based on the estimated correlations. We assume that the kernels $K_i, i=1,2$  satisfy  standard regularity  conditions: (i) $K_i(0)=1$, (ii) $K_i$ is continuous and bounded (iii) there is $c>0$ such that $K_i(u)=0$ if $|u|>c$. It can be shown if $w_1=w_1(n)\to \infty, w_1(n)/n^{1/2}\to 0$ as $n\to \infty$, then $\hat{\mu}_n=\mu+o_P(H^{-1/2})$ under Assumption \ref{as-1}. If $w_2=w_2(H)\to \infty, w_2(H)/H\to 0$, as $H\to \infty$, then one can show that $\hat{\sigma}_n^2\to \sigma^2$ in probability as $n\to \infty$. Hence Theorem \ref{main} yields
\begin{equation}\label{vdef}
{V}_{n,H}=\frac{n\hat{T}_{n,H}-(2H+1)\hat{\mu}_n}{(2H+1)^{1/2}\hat{\sigma}_n}\;\;\;\stackrel{{\mathcal D}}{\to}\;\;\;{\mathcal N},
\end{equation}
where ${\mathcal N}$ stands for a standard normal random variable.
\subsection{Simulation study}\label{simsec}
We now turn to a small simulation study which compares the limit result in Theorem \ref{main} to the finite sample behavior of $\hat{T}_{n,H}$. In order to describe the data generating processes (DGP's) used below, let $\{W_{X,i}\}_{i=1}^{n}$ and $\{W_{Y,i}\}_{i=1}^{n}$ be independent sequences of iid standard Brownian motions on $[0,1]$. Two DGP's were considered which each satisfy the conditions of Theorem \ref{main}:
\begin{align*}
&\hspace*{-2.5 cm}\mbox{IID: two samples of independent Brownian motions, i.e. }
X_i(t)=W_{X,i}(t),\\ % 1\le i \le n, t\in[0,1],\\
&\hspace{-1.1 cm}
Y_i(t)=W_{Y,i}(t),
1\le i \le n,\; t\in[0,1].
\end{align*}

and
\begin{align*}
&\mbox{FAR}_q\mbox{(1): two samples which follow a stationary functional autoregressive model of order}\\
&\hspace{1.5 cm}\mbox{ one, i.e.\ for }  1\le i \le n,\; t\in[0,1],\\
&\hspace{1.7 cm}X_i(t)=\int \psi_q(t,u)X_{i-1}(u)du+W_{X,i}(t),\quad %\\%\;\; 1\le i \le n, \;\;t\in[0,1],\\
Y_i(t)=\int \psi_q(t,u)Y_{i-1}(u)du+W_{Y,i}(t),  \\
&\hspace{1.5 cm}\mbox{ where } \psi_q(t,u)=q\min(t,u). \\
\end{align*}

The motivation behind using Brownian motions for the error sequences in each of the DGP's is due to our application in Section \ref{appsec} to cumulative intraday returns which often look like realizations of Brownian motions, see Figure \ref{lapcurves} below. A precise but somewhat technical condition for the existence of a stationarity solution in the FAR$_q$(1) model is given in Bosq (2000), however it is sufficient to assume that the operator norm of $\psi$ is less than one. In our case we take $q=.75, 1.5, 2.25$ so that $\|\psi\|\approx .25, .5, .75$, respectively. To simulate the stationary FAR$_q$(1) process we used a burn in sample of size 100 starting from an independent innovation. In order to compute the statistic
$V_{n,H}$ of \eqref{vdef}
we must select the kernels $K_1$ and $K_2$ as well as the windows $w_1$ and $w_2$  used to compute the estimates $\hat{\mu}_n$ and $\hat{\sigma}_n^2$. For our analysis we used Bartlett kernels

$$
K_1(t)=K_2(t)=\left\{
\begin{array}{ll}
1-|t|,\quad &0\leq |t| <1\\
0,\quad &|t|\geq 1,
\end{array}
\right.
$$
and windows $w_1(r)=w_2(r)=\lf r^{1/4} \rf$. The simulations results below were repeated for several other common choices of kernel functions including the Parzen and flat--top kernels, as well as different choices for the window parameters. The changes in the results were negligible for different choices of the kernel functions. The results were also stable over window choices between $w_i(r)=\lf r^{1/4}\rf$ to $w_i(r)=\lf r^{1/2}\rf$ which constitute the usual range of acceptable windows. Each DGP was simulated 1000 times independently for various values of $n$ and $H$, and the percentage of test statistics $V_{n,H}$ which exceeded the 10, 5, and 1 percent critical values of the standard normal distribution are reported in Table \ref{tab1}.

{\small
\begin{table}[pht!]
\begin{tabular}{c  c | c  c  c | c  c  c | c  c  c | c  c  c   }
\multicolumn{2}{c}{DGP} & \multicolumn{3}{c}{{IID}} &
\multicolumn{3}{c}{{$\mbox{FAR}_{.75}(1)$}} &
\multicolumn{3}{c}{{$\mbox{FAR}_{1.5}(1)$}}&
\multicolumn{3}{c}{{$\mbox{FAR}_{2.25}(1)$}}  \\
$N$ & H & 10\% & 5\% & 1\% & 10\% & 5\% & 1\% & 10\% & 5\% & 1\% & 10\% &
5\% & 1\% \\
\hline
 50  & 2 & 13.8 &8.0  & 2.7  &9.7&6.1&2.8  &15.2 &10.7 & 4.8 &24.3&17.4&
8.0   \\
     & 5 & 14.6 &7.8  & 3.1  &11.4&6.3&2.9  &16.4 &11.1 & 4.1
&26.0&20.7&11.8    \\
     & 7 & 13.2 &7.6  & 3.8  &10.2&6.0&2.1  &15.5 &10.4&  5.0
&33.0&25.4&13.7    \\
\hline
 100 & 3 & 11.5& 7.0  & 2.7  &9.5&6.0&2.4   &12.4 & 7.0& 3.1
&26.8&20.1&13.3    \\
     & 5 & 10.3 & 5.1  & 2.1 &8.5&4.7&0.8   &12.8 & 7.2 & 3.8
&26.4&22.5&15.8   \\
     & 10& 9.8 & 3.9  & 1.3  &8.7&3.8&0.9   &11.4 & 6.5 & 2.7
&32.0&27.3&20.1  \\
\hline
200 & 4 & 10.0& 6.6   & 2.5  &10.2&5.8&2.1  &10.9 & 6.4& 3.0
&16.5&12.8&7.2   \\
    & 7 & 8.7 & 4.8   & 1.2  &11.0&7.1&2.5  &10.5 & 6.0& 1.9
&21.4&17.7&12.6     \\
    & 15& 9.8 & 5.2   & 1.3  &8.4&4.3&1.7   &10.4 & 7.0& 2.7
&25.5&20.6&14.7     \\
\hline
300 & 5 & 10.3 &5.4  & 1.2   &9.7&6.0&2.1  &9.7 & 5.5  & 2.2
&15.3&10.5&6.8    \\
    & 10& 10.5 &5.7  & 1.4   &10.0&5.7&2.3  &10.4 & 5.3 &1.9   &21.0
&18.1&12.7    \\
    & 17& 9.8  &4.7  & 1.1   &10.1&5.1&1.2  &10.9 & 5.7 & 2.4  &26.4
&22.5&16.4     \\
\end{tabular}
\caption{The percentage of 1000 simulated values of $V_{n,H}$ which
exceeded the 10\%, 5\%, and 1\% critical values of the standard normal
distribution for the DGP's IID and $\mbox{FAR}_q(1)$. }
\label{tab1}
\end{table}
}

Based on this data we have the following remarks:
\begin{itemize} %\itemsep0em
\item[(i)] When the level of dependence within the processes  is not too strong the tail of the distribution of $V_{n,H}$ appears to be approaching that of the standard normal as $n\to \infty$.
\item[(ii)] There appears to be little effect in the results due to the choice of $H$ when the dependence is weak. Even for values of $H$ close to $n^{1/2}$ we still observe finite sample distributions which are close to the limit.
\item[(iii)] When the dependence is strong , i.e.\ $q=2.25$, the tail of the distribution of $V_{n,H}$ becomes significantly heavier than that of the standard normal. This effect is worsened by increasing $H$.
\item[(iv)] In the context of hypothesis testing, $V_{n,H}$ achieves good size even when the temporal dependence is weak to moderate as long as the sample size is large. If strong dependence is suspected in the data then the distribution of $V_{n,H}$ may be skewed to the right.
\end{itemize}

\subsection{Application to cumulative intraday stock returns} \label{appsec}

\begin{figure}
\centering \caption{Five functional data
    objects constructed from the 1-minute average price of
    Apple stock. The vertical lines separate days. }
\vspace*{-8mm}
\includegraphics[scale=.45, width=4in,angle=0]{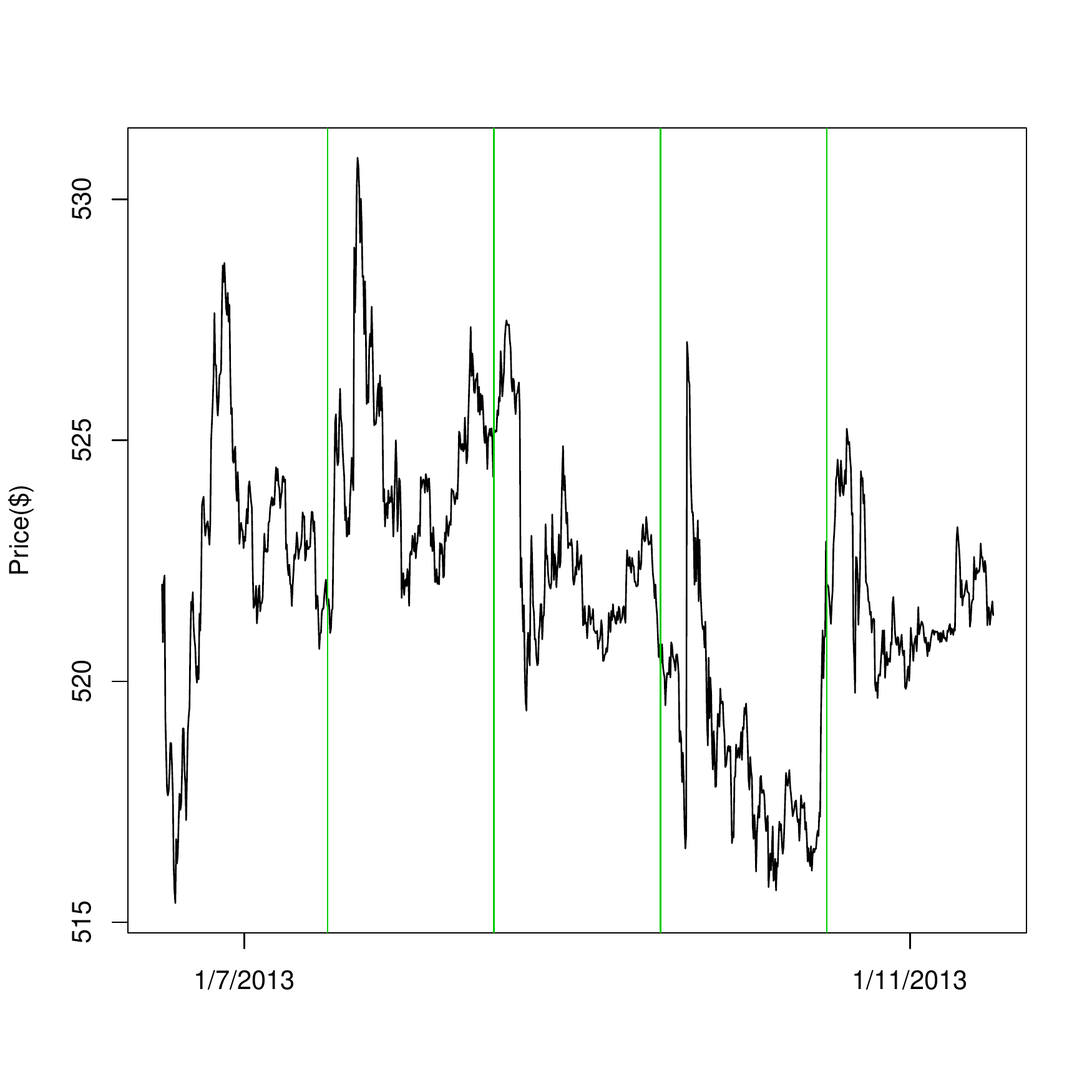}
\label{apcurves}
\end{figure}
A natural example of functional data are intraday price curves constructed from densely observed price records of an asset. Five of such curves constructed from the 1 minute average stock price of Apple are shown in Figure \ref{apcurves}. Several authors have studied the analysis of the shapes of price curves, see M\"uller et al (2011) and Kokoszka et al (2013) for a small sample of such work. In our application we use \eqref{vdef} to test for correlation between samples of curves constructed from pairs of 20 of the highest trade volume stocks listed on the New York Stock Exchange. The data we consider was obtained from  {\tt www.nasdaq.com}  and {\tt  www.finam.com}  and consists of the 1 minute average stock prices of the stocks listed in Table \ref{tab3} from January 1st, 2013 to December 11, 2013, which comprise 254  daily records per stocks. Since the test statistic is based on integrals of the sampled curves its value changes little due to the way the curves are constructed from the original observations, and thus we simply used linear interpolation. It is argued in Horv\'ath et al (2014) that price curves themselves are typically not stationary and thus must be suitably transformed before analysis assuming stationarity may be performed. Scaler price data is typically transformed to appear closer to stationary by taking the log differences, and in our example we employ a similar technique for functional data which was proposed in Gabrys et al (2010):

\begin{defi} \label{def:icr}
Suppose $P_{n}(t_j), n=1,\ldots, N, j=1,\ldots, m$, is the price of
a financial asset at time $t_j$ on day $n$. The functions
\[
R_{n}(t_j) = 100[\ln P_{n}(t_j)-\ln P_{n}(t_1)], \ \ \ \
j = 1, 2,\ldots, m,\ \ \ \
n =1, \ldots, N,
\]
are called the {\em  cumulative intraday returns} (CIDR's).
\end{defi}

{\small
\begin{table}[pht!]
\begin{tabular}{c   | c  | c | c  }
Company Name & Ticker Symbol & Company Name & Ticker Symbol \\
\hline
Apple & AAPL &  AT\&T & ATT  \\
Bank of America & BAC & Boeing & BA \\
Chevron & CVX & Cisco Systems & CSCO \\
Citigroup & C & Coca Cola & KO \\
DuPont & DD & Exxon Mobil & XOM \\
Google & GOOG & HP    & HPQ \\
IBM & IBM & Intel & INTC  \\
McDonalds & MCD & Microsoft & MSFT \\
Verizon & VZ & Walmart & WMT \\
Disney & DIS & Yahoo & YHOO
\end{tabular}
\caption{Company names and ticker symbols for the 20 stocks which were compared. }
\label{tab3}
\end{table}
}

The CIDR's of the five curves in Figure \ref{apcurves} are shown in Figure \ref{lapcurves}. It is apparent that the overall shapes of the price curves are retained by the CIDR's and, since the CIDR's always start from zero, level stationarity is enforced. A more rigorous argument for the stationarity of the CIDR's is given in Horv\'ath et al (2014).

\begin{figure}
\centering \caption{Five cumulative intraday return curves constructed from the price curves displayed in Figure \ref{apcurves}. }
\vspace*{-8mm}
\includegraphics[scale=.45, width=4in,angle=0]{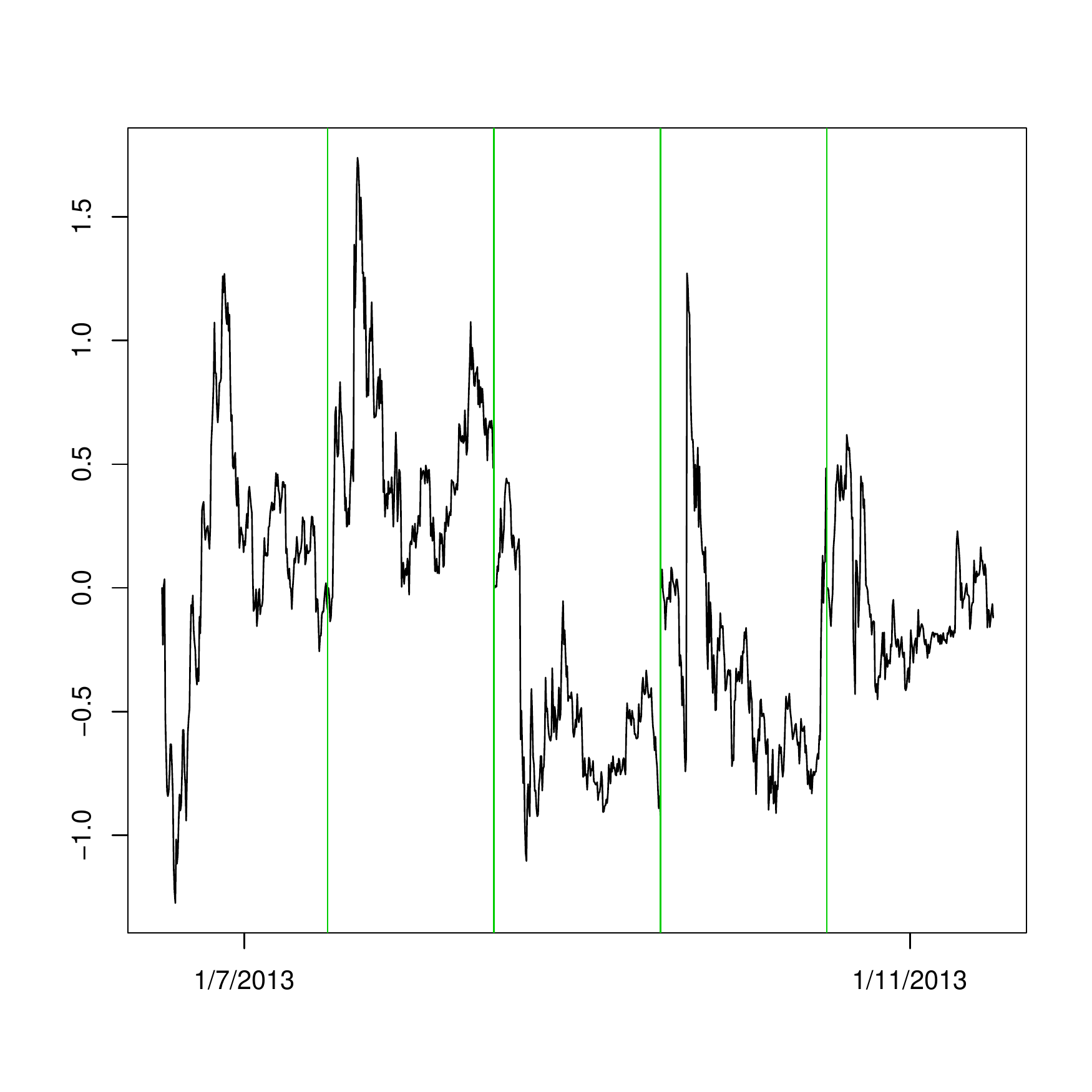}
\label{lapcurves}
\end{figure}

Based on the 20 stocks we used there are 190 pairs of CIDR samples for comparison. For each pair the test statistic $V_{n,H}$ was computed as outlined in Section \ref{simsec} using $w_1(r)=w_2(r)=\lf r^{1/4}\rf$ and $H=4\approx (254)^{1/4}$. An approximate $p$--value of the test was then calculated by taking $1-\Phi(V_{n,H})$, where $\Phi$ is the standard normal distribution function. Of the 190 $p$--values computed 70\% were less than .05 indicating that   there is strong correlation between the CIDR's of most pair of stocks listed on the NYSE. More insight can be achieved by looking at individual comparisons. Table \ref{tab2} contains the $p$--values of all comparisons of a subset of 9 of the stocks. Some clusters are  apparent, like the technology companies (AAPL,IBM,GOOG), energy companies (CVX,XOM) and service companies (WMT,DIS,MCD) for which all of the within group comparisons give effectively zero $p$-values. Also many of the comparisons between the tech companies and energy companies yielded large $p$--values, indicating that these groups of CIDR's  exhibit little cross correlation. Similar clusters became apparent among the rest of the sample of 20 stocks, and using cluster analysis methods based on the values of $V_{n,H}$ could lead to further insights. For example, one could say that stocks A and B are similar if the value of
$V_{n,H}$ computed between the CIDR's from A and B exceeds a certain
threshold, i.e.\    stocks A and B are highly correlated. This will not lead to perfect clustering, however, using
correlation clustering algorithms as in Kriegel et al (2009), the number
of similarities across clusters or dissimilarities within clusters can be
minimized.

{\small
\begin{table}[pht!]
\begin{tabular}{l   | c  c  c  c  c  c c c }
IBM  &.000 &     &      &      &      &      &      &\\
GOOG &.001 &.000 &      &      &      &      &      &\\
XOM  &.541 &.612 & .051 &      &      &      &      &\\
CVX  &.531 &.521 & .003 & .000 &      &      &      &\\
C    &.329 &.007 & .006 & .008 & .001 &      &      &\\
WMT  &.602 &.358 & .000 & .001 & .000 & .000 &      &\\
DIS  &.221 &.002 & .000 & .014 & .000 & .000 & .001 & \\
MCD  &.000 &.018 & .457 & .185 & .000 & .133 & .000 & .043 \\
\hline
     &AAPL &IBM  & GOOG & XOM & CVX   & C &  WMT & DIS
\end{tabular}
\caption{Approximate $p$--values for a test of $H_0$ based on $V_{n,H}$ of \eqref{vdef} for all pairwise comparisons of the 9 stocks AAPL, IBM, GOOG, XOM, CVX, C, WMT, DIS, and MCD.}
\label{tab2}
\end{table}
}

\section{Proof of Theorem \ref{main}}\label{proofs}\setcounter{equation}{0}
The proof of Theorem \ref{main} requires several steps. First we claim that it is enough to prove Theorem \ref{main} for the square integrals of the correlations without centering by the sample means. In the second step we argue that $\hat{T}_{n,h}$ can be approximated  with the sum of integrated squared correlations of $m$--dependent random functions if $m$ and $n$ are both large enough. The last step is the proof of the central limit theorem for the sum of integrated squared correlations of $m$--dependent functions for every fixed $m$, when $n\to \infty.$ This is established using a blocking argument.\\
It is easy to see that $\hat{T}_{n,H}$ does not depend on the means of the observations, and so  we can assume without loss of generality that
\beq\label{zero}
EX_0(t)=0\quad \mbox{and}\quad EY_0(t)=0.
\eeq
In order to simplify the calculations   we define

$$
\tilde{C}_{n,h}(t,s)\displaystyle= \frac{1}{n}\sum_{i=1}^{n} X_i(t)Y_{i+h}(s),\quad -H\leq h \leq H
$$
and
$$
\tilde{T}_{n,H}=\sum_{h=-H}^H \tilde{\xi}_{n,h}\quad \quad\mbox{with}\quad \quad  \tilde{\xi}_{n,h}=\intt \tilde{C}_{n,h}^2(t,s)ds.
$$

\begin{lemma}\label{nomean} If Assumptions \ref{as-2}--\ref{as-1}  and \eqref{zero} hold, then we have
$$
E[n(\hat{T}_{n,H}-\tilde{T}_{n,H})]^2=o(H)\;\;\;\mbox{as}\;\;\;n\to \infty.
$$
\end{lemma}
\begin{proof} This claim can be proven by standard arguments so the details are omitted.
\end{proof}
For every $m\geq 1$ we define   according to Assumption \ref{as-3}
\beq\label{xm-def}
X_{i,m}=g_\X(\epsilon_{i},\epsilon_{i-1},...,\epsilon_{i-m+1}, \beps_{i,m}^*),
\eeq
with $\beps^*_{i,m}=(\epsilon_{i,m, i-m}^*,\epsilon_{i,m,
i-m-1}^*,\ldots),$  where the $\epsilon^*_{j,\ell,k}$' are independent copies of $\epsilon_{0}$ and
independent of $\{\epsilon_{j}, \bar{\epsilon}_{j},
-\infty<j<\infty\}.$ The random function $Y_{i,m}$ is defined analogously by
\beq\label{ym-def}
Y_{i,m}=g_\Y(\bar{\epsilon}_{i},\bar{\epsilon}_{i-1},...,\bar{\epsilon}_{i-m+1}, \bar{\beps}_{i,m}^*),
\eeq
with $\bar{\beps}^*_{i,m}=(\bar{\epsilon}_{i,m, i-m}^*, \bar{\epsilon}_{i,m,
i-m-1}^*,\ldots),$  where the $\bar{\epsilon}^*_{j,\ell,k}$' are independent copies of $\bar{\epsilon}_{0}$ and
independent of $\{\epsilon_{j}, \bar{\epsilon}_{j},\epsilon^*_{j,\ell,k},
-\infty<j,k,\ell<\infty\}.$ It follows from the definition that both $X_{i,m}, -\infty<i<\infty$ and $Y_{i,m}, -\infty<i<\infty$ are stationary $m$--dependent sequences, independent of each other. Also, for every $i$ and $m$ we have that $X_{i,m}=X_0$ and $Y_{i,m}=Y_0$ in distribution.
Next we introduce

$$
\bar{C}_m(t,s)=\bar{C}_{n,h,m}(t,s)=\displaystyle \frac{1}{n}\sum_{i=1}^{n} X_{i,m}(t)Y_{i+h,m}(s),\quad -H\leq h \leq H
$$
and
$$
\bar{T}_{n,H,m}= \sum_{h=-H}^H \bar{\xi}_{n,h,m}\quad\quad \mbox{with}\quad \quad \bar{\xi}_{n,h,m}=\intt \bar{C}_{n,h,m}^2(t,s)ds.
$$
We also use the notation
$$
J_\X(h)=\sum_{i=h}^\infty (E\|X_0-X_{0,i}\|^4)^{1/4},\;\;
G_\X(h)=\sum_{i=h}^\infty i (E\|X_0-X_{0,i}\|^2)^{1/2}
$$
and similarly
 $$
J_\Y(h)=\sum_{i=h}^\infty (E\|Y_0-Y_{0,i}\|^4)^{1/4},\;\;G_\Y(h)=\sum_{i=h}^\infty i (E\|Y_0-Y_{0,i}\|^2)^{1/2}
.
$$
in the rest of the proofs.

\begin{lemma}\label{mapp} If Assumptions \ref{as-2}--\ref{as-1} and \eqref{zero} hold, then we have

\sa
\lim_{m\to \infty}\displaystyle \underset{n\to\infty}{\mbox{\rm limsup}}E\left[\frac{n}{H^{1/2}}\left\{(\tilde{T}_{n,H}-E\tilde{T}_{n,H})-(\bar{T}_{n,H,m}-E\bar{T}_{n,H,m})\right\}\right]^2=0.
\end{align*}
\end{lemma}
\begin{proof} By the stationarity of $\{(X_i, X_{i,m}), -\infty <i<\infty\}$ and $\{(Y_i, Y_{i,m}), -\infty <i<\infty\}$ we have
\sa
E&\left\{(\tilde{T}_{n,H}-E\tilde{T}_{n,H})-(\bar{T}_{n,H,m}-E\bar{T}_{n,H,m})\right\}^2\\
&\hspace{.4 cm}\leq  (2H+1)\mbox{var}(\tilde{\xi}_{n,0}-\bar{\xi}_{n,0,m})+ 4H\sum_{h=1}^{2H}|\mbox{cov}(\tilde{\xi}_{n,0}-\bar{\xi}_{n,0,m},\tilde{\xi}_{n,h}-\bar{\xi}_{n,h,m})|.
\end{align*}
It follows from Assumptions \ref{as-3}--\ref{as-5} that for every fixed $h$
\beq\label{m-1}
\lim_{m\to \infty}\displaystyle \underset{n\to \infty}{\mbox{\rm limsup }}n^2\mbox{cov}(\tilde{\xi}_{n,0}-\bar{\xi}_{n,0,m},\tilde{\xi}_{n,h}-\bar{\xi}_{n,h,m})=0.
\eeq
For any $M\geq 1$ we write
\sa
\sum_{h=1}^{2H}&|\mbox{cov}(\tilde{\xi}_{n,0}-\bar{\xi}_{n,0,m},\tilde{\xi}_{n,h}-\bar{\xi}_{n,h,m})|\\
&\leq \sum_{h=1}^{M}|\mbox{cov}(\tilde{\xi}_{n,0}-\bar{\xi}_{n,0,m},\tilde{\xi}_{n,h}-\bar{\xi}_{n,h,m})|\\
&\hspace{0.5 cm}+\sum_{h=M+1}^H \left(|\mbox{cov}(\tilde{\xi}_{n,0}, \tilde{\xi}_{n,h})| +|\mbox{cov}(\bar{\xi}_{n,0,m},\bar{\xi}_{n,h,m})|+2|\mbox{cov}(\tilde{\xi}_{n,0},\bar{\xi}_{n,h,m})|\right).
\end{align*}
It follows from \eqref{m-1} that for all $M\geq 1$
\beq\label{m-1/2}
\lim_{m\to \infty}\displaystyle \underset{n\to \infty}{\mbox{\rm limsup }}n^2\sum_{h=1}^{M}|\mbox{cov}(\tilde{\xi}_{n,0}-\bar{\xi}_{n,0,m},\tilde{\xi}_{n,h}-\bar{\xi}_{n,h,m})|=0.
\eeq
We prove in Appendix  \ref{sec-mom} that there is a constant $C$ depending only on $E\|X_0\|^4, E\|Y_0\|^4, J_\X(0),$ $ J_\Y(0), G_\X(0) $  and $G_\Y(0)$ such that for all $h\geq 1$
\beq\label{m-2}
n^2|\mbox{cov}(\tilde{\xi}_{n,0}, \tilde{\xi}_{n,h})|\leq C\left( J_\bX(h)+J_\bY(h)+\frac{h^2}{n}  \right).
\eeq
Hence we get that
\beq\label{m-3}
n^2\sum_{h=M+1}^H |\mbox{cov}(\tilde{\xi}_{n,0}, \tilde{\xi}_{n,h})| \leq C\left(\sum_{h=M+1}^\infty (J_\X(h)+J_\Y(h))+\frac{H^3}{n}
\right)
\eeq
and therefore by Assumptions \ref{as-3}, \ref{as-4} and \ref{as-1}
\beq\label{m-4}
\lim_{M\to \infty}\displaystyle \underset{n\to \infty}{\mbox{\rm limsup }}n^2\sum_{h=M+1}^H |\mbox{cov}(\tilde{\xi}_{n,0}, \tilde{\xi}_{n,h})|=0.
\eeq
 Following the proof of Lemma \ref{l-lem} one can verify that
\beq\label{m-5}
n^2|\mbox{cov}(\tilde{\xi}_{n,0}, \bar{\xi}_{n,h,m})|\leq C\left( J_\X(h)+J_\Y(h)+\frac{h^2}{n}  \right),
\eeq
and
\beq\label{m-6}
n^2|\mbox{cov}(\bar{\xi}_{n,0,m}, \bar{\xi}_{n,h,m})|\leq C\left( J_\X(h)+J_\Y(h)+\frac{h^2}{n}  \right)
\eeq
with the the same $C$ as in \eqref{m-2}. We wish to note that the proofs of \eqref{m-5} and \eqref{m-6} are simpler than that of \eqref{m-2} due to the $m$--dependence of $X_{i,m}$ and $Y_{j,m}$. Using  \eqref{m-5} and \eqref{m-6} we get that
$$
\lim_{M\to \infty}\displaystyle \underset{n\to \infty}{\mbox{\rm limsup }}n^2\sum_{h=M+1}^H \left(|\mbox{cov}(\bar{\xi}_{n,0,m}, \bar{\xi}_{n,h,m})|
+2|\mbox{cov}(\tilde{\xi}_{n,0}, \bar{\xi}_{n,h,m})|
\right)=0,
$$
and therefore Lemma \ref{mapp} follows from \eqref{m-1/2} and \eqref{m-2}.
\end{proof}

According to Lemma \ref{mapp}, it is enough to prove Theorem \ref{main} for $m$ dependent variables.

\begin{lemma}\label{m-cent} If Assumptions \ref{as-2}--\ref{as-1}  and \eqref{zero} hold, then for every $m\geq 1$ we have that
$$
\frac{n\bar{T}_{n,H,m}-(2H+1)\bar{\mu}}{(2H+1)^{1/2}\bar{\sigma}}\;\;\;\stackrel{{\mathcal D}}{\to}\;\;\;{\mathcal N},
$$
where ${\mathcal N}$ denotes a standard normal random variable.
\end{lemma}

\begin{proof}
First  we define  $N=H^3\log H$, and the sets
$B(k,\ell)=\{(i,j): (k-1)(N+3H)+1\leq i\leq kN+3(k-1)H, (\ell-1)(N+3H)+1\leq j\leq N+3(\ell-1)H\}, 1\leq k,\ell\leq M$, where $M$ is defined
as the largest integer satisfying $MN+3(M-1)H\leq n$, i.e.\ $M=\lf n/(N+3H)\rf$.  Let
$$
\bar{\eta}_{i,j}=\int X_{i,m}(t)X_{j,m}(t)dt\quad \mbox{and}\quad \bar{\theta}_{i,j}=\int Y_{i,m}(t)Y_{j,m}(t)dt
$$
and define
$$
e_{k,\ell}=\sum_{(i,j)\in B(k,\ell)}\sum_{h=-H}^H\bar{\eta}_{i,j}\bar{\theta}_{i+h,j+h}.
$$
Since the random variables $\bar{\eta}_{i,j}, \bar{\theta}_{i,j}$ are constructed from  $m$--dependent random functions  we can assume  by letting  $H>m\geq 1$, that $e_{k,\ell}, 1\leq k,\ell\leq M$ are independent and identically distributed. Using Petrov (1995, p.\ 58) we get that
$$
E|e_{k,\ell}-Ee_{k,\ell}|^{2+\delta/2}\leq (2H+1)^{1+\delta/2}\sum_{h=-H}^HE\left| \sum_{(i,j)\in B(k,\ell)}[
\bar{\eta}_{i,j}\bar{\theta}_{i+h,j+h}-E\bar{\eta}_{i,j}\bar{\theta}_{i+h,j+h}]\right|^{2+\delta/2}.
$$
Using the stationarity and the $m$--dependence of the $X_{i,m}$'s and the $Y_{j,m}$'s we obtain that
$$
E\left| \sum_{(i,j)\in B(k,\ell)}[
\bar{\eta}_{i,j}\bar{\theta}_{i+h,j+h}-E\bar{\eta}_{i,j}\bar{\theta}_{i+h,j+h}]\right|^{2+\delta/2}\leq CN^{2+\delta/2},
$$
where $C$ only depends on $m$, $E\|X_0\|^{4+\delta}$ and $E\|Y_0\|^{4+\delta}$.
Also, it follows from Appendix \ref{sec-var} that there is $c_0>0$ such that
$$
\mbox{var}(e_{k,\ell})\geq c_0N^2H,
$$
since $H^3/N\to 0$. Thus by Assumption \ref{as-1}  we get that
\sa
\frac{\left(\sum_{1\leq k ,\ell\leq M}E|e_{k,\ell}-Ee_{k,\ell}|^{2+\delta/2}\right)^{1/(2+\delta/2)}}{\left(\sum_{1\leq k ,\ell\leq M} \mbox{var}(e_{k,\ell})\right)^{1/2} }=O\left(\frac{M^{2/(2+\delta/2)}HN}{MNH^{1/2}}\right)=o(1),
\end{align*}
and so by Lyapunov's theorem (cf.\ Petrov (1995, p.\ 126)) we conclude
$$
\frac{\sum_{1\leq k,\ell \leq M}(e_{k,\ell}-Ee_{k,\ell})}{M\mbox{var}(e_{1,1})^{1/2}}\;\;\stackrel{{\mathcal D}}{\to}\;\;{\mathcal N},
$$
where ${\mathcal N}$ denotes a standard normal random variable.\\
Next we define the small blocks which do not contribute to the limit. Let $C(k,\ell)=\{(i,j): kN+3(k-1)H+1\leq i\leq kN+3kH, \ell N+3(\ell-1)H+1\leq j\leq \ell N+3\ell H\}, 1\leq k,\ell\leq M-1$. Using Petrov (1995, p.\ 58)  and stationarity we conclude that
\beq\label{varH}
\mbox{var}\left( \sum_{(i,j)\in C(k,\ell)}\sum_{h=-H}^H\bar{\eta}_{i,j}\bar{\theta}_{i+h,j+h}   \right)\leq C_1 H^2\mbox{var}\left(\sum_{(i,j)\in C(k,\ell)}
\bar{\eta}_{i,j}\bar{\theta}_{i,j} \right)\leq C_2H^4
\eeq
with some constants $C_1$ and $C_2$. Since $H>m$, by independence we obtain that
\begin{align}\label{varH-1}
\mbox{var}\left(\sum_{1\leq k,\ell<M} \sum_{(i,j)\in C(k,\ell)}\sum_{h=-H}^H\bar{\eta}_{i,j}\bar{\theta}_{i+h,j+h}\right)
&=O\left( {M^2 H^4}\right)\\
&=O(n^2N^{-2}H^4)=o(n^2H)\notag
\end{align}
on account of $H^{3/2}/N\to 0$. Let $D_1(k,\ell)=\{(i,j): (k-1)(N+3H)+1\leq i\leq kN+3(k-1)H, \ell N+3(\ell-1)H+1\leq j\leq \ell N+3\ell H\}, 1\leq k\leq M,1\leq \ell\leq M-1$. We can divide $D_1(k,\ell)$ into $\lf N/H\rf$ blocks of size $3H\times 3H$ and one additional smaller block. The sums of $\sum_{h=-H}^H\bar{\eta}_{i,j}\bar{\theta}_{i+h,j+h}$ over these blocks give 2 dependent variables so by \eqref{varH} we get that
$$
 \mbox{var}\left(\sum_{(i,j)\in D_1(k,\ell)}\sum_{h=-H}^H\bar{\eta}_{i,j}\bar{\theta}_{i+h,j+h}\right)
=O\left(\frac{N}{H}H^4\right).
$$
Thus we have
\begin{align}\label{varH-2}
\sum_{1\leq k,\ell<M}&\sum_{(i,j)\in D_1(k,\ell)}\sum_{h=-H}^H(\bar{\eta}_{i,j}\bar{\theta}_{i+h,j+h} -E\bar{\eta}_{i,j}\bar{\theta}_{i+h,j+h})\\
&=O_P\left(MN^{1/2}H^{3/2}\right)
=O_P\left( nN^{-1/2}H^2\right)=o_P(nH^{1/2})\notag
\end{align}
on account of $H^3/N\to 0$. The same arguments give
\sa
\sum_{1\leq k,\ell<M}\sum_{(i,j)\in D_2(k,\ell)}\sum_{h=-H}^H(\bar{\eta}_{i,j}\bar{\theta}_{i+h,j+h} -E\bar{\eta}_{i,j}\bar{\theta}_{i+h,j+h})
=o_P(nH^{1/2}),
\end{align*}
where $D_2(k,\ell)=\{(i,j): kN+3(k-1)H+1\leq i \leq N+3kH, (\ell-1)(N+3H)+1\leq j \leq \ell N+3(\ell-1)H\}, 1\leq k\leq M-1, 1\leq \ell\leq M.$ Let $ D=\{1\leq i,j \leq n\} \backslash (\cup _{1\leq k,\ell\leq M}C_{k,\ell}\cup_{1\leq k \leq M, 1\leq \ell \leq M-1}D_1(k,\ell) \cup_{1\leq k \leq M-1, 1\leq \ell \leq M}D_2(k,\ell)\cup_{1\leq k, \ell\leq M}B(k,\ell))$. Repeating the proofs of \eqref{varH-1} and \eqref{varH-2} one can verify that
\sa
\sum_{(i,j)\in D}\sum_{h=-H}^H(\bar{\eta}_{i,j}\bar{\theta}_{i+h,j+h} -E\bar{\eta}_{i,j}\bar{\theta}_{i+h,j+h})
=o_P(nH^{1/2}),
\end{align*}
completing the proof of Lemma \ref{m-cent}.
\end{proof}
\medskip

\medskip
\noindent
{\bf {\sc Summary:}} The problem of testing for independence between two finite dimensional
time series has been studied extensively in the literature and most
available methods are based on cross covariance estimators. In this paper
we established the asymptotic normality of a test statistic constructed
from the operator norms of empirical cross covariance operators using
functional data. This asymptotic result supplies a test of the null
hypothesis that two functional time series are independent. The rate at
which the asymptotic result is achieved as the sample size increases was
investigated using a Monte Carlo simulation study and the results showed
that the limit approximation was accurate for moderate to large sample
sizes. Finally we illustrated an application of our limit result to
testing for independence between cumulative intraday return curves
constructed from the closing prices of stocks.

\medskip
\noindent
{\bf {\sc Acknowledgements:}} We are grateful to two anonymous referees for their careful reading of our
paper and for their helpful insights and remarks.

\medskip
\appendix
\section{Bound for covariances}\label{sec-mom}\setcounter{equation}{0}
Let
$$
\eta_{\ell,k}=\int X_\ell(t)X_k(t)dt,\quad \theta_{\ell,k}=\int Y_\ell(t)Y_k(t)dt
$$
and
$$
a(\ell-k)=E\eta_{\ell,k},\quad b(\ell-k)=E\theta_{\ell,k}.
$$
Clearly,  $a(x)=a(-x)$ and $b(x)=b(-x)$.
Let
$$
V_n(h)=\mbox{\rm cov}\left( \sum_{i=1}^n\sum_{j=1}^n \eta_{i,j}\theta_{i, j}, \sum_{k=1}^n \sum_{\ell=1}^n\eta_{k,\ell}\theta_{k+h, \ell+h} \right).
$$
In addition to $J_\X, J_\Y, G_\X$ and $G_\Y$ defined in Section \ref{proofs} we also use
$$
L_\X(h)=\sum_{i=h}^\infty (E\|X_0-X_{0,i}\|^2)^{1/2}\; \;\mbox{and}\;\;L_\Y(h)=\sum_{i=h}^\infty (E\|Y_0-Y_{0,i}\|^2)^{1/2}.
$$
\begin{lemma}\label{l-lem} If  Assumptions \ref{as-2}--\ref{as-1} and  \eqref{zero}  hold, then for all $-H\leq  h \leq H$ we have
\sa
\left|V_n(h)\right|\leq C&\bigl\{n^2(J_\X(h)J_\Y(0)+J_\X(0)J_\Y(h))+n(J_\X(0)J_\Y(0)+G_\X(0)J_\Y(0)+J_\X(0)G_\Y(0))\\
&+nh^2(J_\X(0)+J_\Y(0))+J_\X(0)J_\Y(0))
\bigl\},
\end{align*}
where $C$ only depends on $E\|X_0\|^4$ and $E\|Y_0\|^4$.
\end{lemma}
\begin{proof} It is easy to see that by Assumptions \ref{as-3}--\ref{as-5} we have
\begin{align*}
\left|V_{n,h}\right|&=\left|\sum_{i=1}^n\sum_{j=1}^n\sum_{k=1}^n\sum_{\ell=1}^n\biggl(E\eta_{i,j}\eta_{k,\ell}E\theta_{i,j}\theta_{k+h,\ell+h}
-a(i-j)a(k-\ell)b(i-j)b(k-\ell)\biggl)\right|\\
&\leq n\sum_{k=1}^n\sum_{\ell=1}^n\left|E\eta_{0,0}\eta_{k,\ell}E\theta_{0,0}\theta_{k+h,\ell+h}
-a(0)a(k-\ell)b(0)b(k-\ell)\right|\\
&+2n \sum_{j=1}^n\sum_{k=1}^n\sum_{\ell=1}^n\left| E\eta_{0,j}\eta_{k,\ell}E\theta_{0,j}\theta_{k+h,\ell+h}-a(j)a(k-\ell)b(j)k(k-\ell)
\right|\\
&\leq 2n \sum_{j=0}^n\sum_{k=1}^n\sum_{\ell=1}^n\left| E\eta_{0,j}\eta_{k,\ell}E\theta_{0,j}\theta_{k+h,\ell+h}-a(j)a(k-\ell)b(j)k(k-\ell)
\right|.
\end{align*}
We proceed by bounding the sums in $j,k,\ell$ over several subsets of the indices whose union covers all possible combinations. \\
Case I. Let $R_1=\{(j,k,\ell): 0\leq j\leq k\leq \ell\leq n\}, R_{1,1}=R_{1,1}(h)=\{(j,k,\ell)\in R_1: j\geq h\}$, $ R_{1,2}=R_{1,2}(h)=\{(j,k,\ell)\in R_1: \ell-k\geq h\}$ and $R_{1,3}=R_{1,3}(h)=R_1 \backslash (R_{1,1}(h)\cup R_{1,2}(h))$. Using Assumption \ref{as-3} and the Cauchy--Schwarz inequality we get for all $j\geq 0$
\begin{align}\label{a-bau}
|a(j)|
&=\left|E\int X_0(t)X_j(t)dt\right|\\
&=\left|\int EX_0(t)((X_j(t)-X_{j,j}(t))dt\right|\notag\\
&\leq \int (EX_0^2(t))^{1/2}(E(X_j(t)-X_{j,j}(t)))^{1/2}dt\notag\\
&\leq \left(  \int EX^2_0(t)dt\int E(X_j(t)-X_{j,j}(t))^2dt      \right)^{1/2}\notag\\
&=(E\|X_j\|^2)^{1/2}(E\|X_j-X_{j,j}\|^2)^{1/2}\notag\\
&=(E\|X_0\|^2)^{1/2}(E\|X_0-X_{0,j}\|^2)^{1/2},\notag
\end{align}
since by construction $X_0$ and $X_{j,j}$ are independent with zero mean. Similarly, for all $j\geq 0$ we have
\begin{align}\label{b-bau}
|b(j)|\leq (E\|Y_0\|^2)^{1/2}(E\|Y_0-Y_{0,j}\|^2)^{1/2}.
\end{align}
By bounds in \eqref{a-bau} and \eqref{b-bau} we conclude
\begin{align}\label{ab-bau}
\frac{1}{n}\sum_{(j,k,\ell)\in R_{1,1}}|a(j)a(\ell-k)b(j)b(\ell-k)|
\leq 4 (E\|X_0\|^2E\|Y_0\|^2)^{3/2}L_\X(h)L_\Y(0).
\end{align}
Assumption \ref{as-3}  implies
\begin{align*}
E\eta_{0,j}\eta_{k,\ell}&= \intt EX_0(t)X_j(t)X_k(s)X_\ell(s)dtds   \\
&= \intt EX_0(t)(X_j(t)X_k(s)X_\ell(s) -  X_{j,j}(t)X_{k,j}(s)X_{\ell,j}(s)    )dtds   \\
&=E\int X_0(t)(X_j(t)-X_{j,j}(t))dt\int X_{k}(s)X_\ell(s)ds\\
&\hspace{.5 cm}+ E\int X_0(t)X_{j,j}(t)dt \int X_k(s)(X_\ell(s)-X_{\ell,j}(s))ds\\
&\hspace{.5 cm}+ E\int X_0(t)X_{j,j}(t)dt\int X_{\ell, j}(s)(X_k(s)-X_{k,j}(s))ds.
\end{align*}
The Cauchy--Schwarz inequality and the stationarity of the $X_\ell$'s gives as in \eqref{a-bau} that
$$
\left|E\int X_0(t)(X_j(t)-X_{j,j}(t))dt\int X_{k}(s)X_\ell(s)ds\right|\leq (E\|X_0\|^4)^{3/4}(E\|X_0-X_{0,j}\|^4)^{1/4}
$$
and similarly
 $$
\left|  E\int X_0(t)X_{j,j}(t)dt \int X_k(s)(X_\ell(s)-X_{\ell,j}(s))ds          \right|\leq (E\|X_0\|^4)^{3/4}(E\|X_0-X_{0,j}\|^4)^{1/4},
$$
$$
\left|  E\int X_0(t)X_{j,j}(t)dt \int X_{\ell, j}(s)(X_k(s)-X_{k,j}(s))ds          \right|\leq (E\|X_0\|^4)^{3/4}(E\|X_0-X_{0,j}\|^4)^{1/4}.
$$
Thus we get
\beq\label{et-1}
|E\eta_{0,j}\eta_{k,\ell}|\leq 3 (E\|X_0\|^4)^{3/4}(E\|X_0-X_{0,j}\|^4)^{1/4}
\eeq
Repeating the arguments above, one can easily verify that
\beq\label{et-2}
| E\theta_{0,j}\theta_{k+h, \ell+h} | \leq 3(E\|Y_0\|^4)^{3/4}(E\|Y_0-Y_{0,\ell-k}\|^4)^{1/4}.
\eeq
Combining \eqref{et-1} and \eqref{et-2} we get
\begin{align}\label{r-11}
&\sum_{(j,k,\ell)\in R_{1,1}}  | E\eta_{0,j}\eta_{k,\ell}  E\theta_{0,j}\theta_{k+h, \ell+h}|
\leq 9 n(E\|X_0\|^4)^{3/4}(E\|Y_0\|^4)^{3/4}J_\X(h)J_\Y(0).
\end{align}
Minor modifications of the proofs of \eqref{ab-bau} and \eqref{r-11} lead to
\begin{align}\label{ab-bau-2}
&\sum_{(j,k,\ell)\in R_{1,2}}|a(j)a(\ell-k)b(j)b(\ell-k)|
\leq nE\|X_0\|^2E\|Y_0\|^2 L_\X(0)L_\Y(h)
\end{align}
and
\begin{align}\label{r-11-2}
&\sum_{(j,k,\ell)\in R_{1,2}}  | E\eta_{0,j}\eta_{k,\ell}  E\theta_{0,j}\theta_{k+h, \ell+h}|
\leq 9 n(E\|X_0\|^4)^{3/4}(E\|Y_0\|^4)^{3/4}L_\X(0)L_\Y(h).
\end{align}
Next we develop bounds for the sum when the indices are in $R_{1,3}$. We define the random functions
${X}^*_{k,k-j}=g_\bX(\epsilon_k, \epsilon_{k-1}, \ldots ,\epsilon_{j+1}, \epsilon^*_j, \epsilon^*_{j-1}\ldots )$ and ${X}^*_{\ell, \ell-j}=g_\bX(\epsilon_\ell, \epsilon_{\ell-1}, \ldots ,\epsilon_{j+1}, \epsilon_j^*,$ $ \epsilon_{j-1}^*, \ldots )$, where $\epsilon_i^*, -\infty <i<\infty$ are iid copies of $\epsilon_0$, independent of $\epsilon_i, -\infty<i<\infty$. It is clear that $(X_k, X_\ell)$ and $({X}^*_{k,k-j}, {X}^*_{\ell, \ell-j})$ have the same distribution and are independent of $(X_0, X_j)$. (Note that ${X}^*_{k,k-j}$ and ${X}^*_{\ell,\ell-j}$ are defined by using the same $\epsilon^*_j$'s, making them different from $X_{k,k-j}$ and $X_{\ell, \ell-j}$ of \eqref{xm-def}).  By Assumption \ref{as-3}  and the Cauchy--Schwarz inequality we have
\begin{align}\label{c-bau}
|E\eta_{0,j}&\eta_{k,\ell}-a(j)a(\ell-k)|\\
&\leq \intt E|X_0(t)X_j(s)(X_{k}(s)X_{\ell}(s)-{X}^*_{k,k-j}(s){X}^*_{\ell, \ell-j}(s))|dtds \notag\\
&\leq (E\|X_0\|^4)^{3/4}((E\|X_k-X^*_{k,k-j}\|^4)^{1/4}+(E\|X_{\ell}-X^*_{\ell,\ell-j}\|^4)^{1/4})\notag\\
&=(E\|X_0\|^4)^{3/4}((E\|X_0-X_{0,k-j}\|^4)^{1/4}+(E\|X_{0}-X_{0,\ell-j}\|^4)^{1/4})\notag
\end{align}
and similarly
 \begin{align*}
|E&\theta_{0,j}\theta_{k+h,\ell+h}-b(j)b(\ell-k)|\\
%&\leq \intt E|Y_0(t)Y_j(s)(Y_{k+h}(s)Y_{\ell+h}(s)-Y_{k+h, k-j}(s)Y_{\ell+h, \ell+h-j}(s))|dtds\\
&\leq (E\|Y_0\|^4)^{3/4}((E\|Y_{0}-Y_{0,k+h-j}\|^4)^{1/4}+(E\|Y_{0}-Y_{0,\ell+h-j}\|^4)^{1/4}).
\end{align*}
Clearly,
\begin{align*}
\sum_{(j,k,\ell)\in R_{1,3}}|E\eta_{0,j}\eta_{k,\ell}E\theta_{0,j}\theta_{k+h,\ell+h}-a(j)a(\ell-k)b(j)b(\ell-k)|
\leq Q_{n,1}+Q_{n,2}+Q_{n,3}
\end{align*}
with
\begin{align*}
Q_{n,1}=\sum_{(j,k,\ell)\in R_{1,3}}|E\eta_{0,j}\eta_{k,\ell}-a(j)a(\ell-k)||b(j)b(\ell-k)|,
\end{align*}

\begin{align*}
Q_{n,2}=\sum_{(j,k,\ell)\in R_{1,3}}|E\theta_{0,j}\theta_{k+h,\ell+h}-b(j)b(\ell-k)||a(j)a(\ell-k)|
\end{align*}
\begin{align*}
Q_{n,3}=\sum_{(j,k,\ell)\in R_{1,3}}|E\eta_{0,j}\eta_{k,\ell}-a(j)a(\ell-k)||E\theta_{0,j}\theta_{k+h,\ell+h}-b(j)b(\ell-k)|.
\end{align*}
Applying \eqref{b-bau} and \eqref{c-bau} we get that
\begin{align*}
Q_{n,1}\leq &(E\|X_0\|^4)^{1/4}E\|Y_0\|^2\biggl(\sum_{(j,k,\ell)\in R_{1,3}}(E\|X_0-X_{0,k-j}\|^4)^{1/4}(E\|Y_0-Y_{0,j}\|^2E\|Y_0-Y_{0,\ell-k}\|^2)^{1/2}\\
& \hspace{1cm}+\sum_{(j,k,\ell)\in R_{1,3}}(E\|X_0-X_{0,\ell-j}\|^4)^{1/4}(E\|Y_0-Y_{0,j}\|^2E\|Y_0-Y_{0,\ell-k}\|^2)^{1/2}\biggl).
\end{align*}
Using the definition of $R_{1,3}$ we conclude
\begin{align*}
&\sum_{(j,k,\ell)\in R_{1,3}}(E\|X_0-X_{0,k-j}\|^4)^{1/4}(E\|Y_0-Y_{0,j}\|^2E\|Y_0-Y_{0,\ell-k}\|^2)^{1/2}\\
&=\sum_{j=1}^h(E\|Y_0-Y_{0,j}\|^2)^{1/2}
\sum_{k=j}^n(\|X_0-X_{0,k-j}\|^4)^{1/4}\sum_{\ell=k}^{k+h}(E\|Y_0-Y_{0,\ell-k}\|^2)^{1/2}\\
&\leq h^2 \max_{1\leq r \leq h}E\|Y_0-Y_{0,r}\|^2J_\X(0)\\
&\leq 2h^2 E\|Y_0\|^2J_\X(0).
\end{align*}
It follows similarly
\begin{align*}
\sum_{(j,k,\ell)\in R_{1,3}}(E\|X_0-X_{0,\ell-j}\|^4)^{1/4}(E\|Y_0-Y_{0,j}\|^2E\|Y_0-Y_{0,\ell-k}\|^2)^{1/2}
\leq 2h^2 E\|Y_0\|^2J_\X(0),
\end{align*}
implying
\begin{align*}
Q_{n,1}\leq 4h^2 (E\|X_0\|^4)^{1/4}(E\|Y_0\|^2)^2J_\X(0).
\end{align*}
Using nearly identical arguments we obtain that
\begin{align*}
Q_{n,2}\leq 4h^2 (E\|Y_0\|^4)^{1/4}(E\|X_0\|^2)^2J_\Y(0)
\end{align*}
and
\begin{align*}
Q_{n,3}\leq 8h^2 (E\|X_0\|^4)^{3/4}E\|Y_0\|^4J_\X(0).
\end{align*}
Case II. Let $R_2=\{(j,k,\ell): 0\leq k \leq \ell\leq j\}, R_{2,1}=R_{2,1}(h)=\{(j,k,\ell)\in R_2 : 0\leq k +h\leq \ell+h\leq j\}, R_{2,2}=R_{2,2}(h)=\{(j,k,\ell)\in R_2 : 0\leq k+h\leq j \leq \ell+h \}$ and$ R_{2,3}=R_{2,3}(h)=\{(j,k,\ell)\in R_2 : 0\leq j\leq k+h\leq \ell+h\}.$ For
$(j,k,\ell)\in R_{2,1}$ we write with the help of Assumption \ref{as-3}
\begin{align*}
|E\eta_{0,j}\eta_{k,\ell}|&=\left|E\intt X_0(t)X_k(s)X_\ell(s)X_j(t)dtds\right|\\
&=\left|E\intt X_0(t)X_k(s)X_\ell(s)(X_j(t)-X_{j,j-\ell}(t)dtds)\right|\\
&\leq ( E\|X_0\|^4)^{3/4}(E\|X_0-X_{0,j-\ell}\|^4)^{1/4}
\end{align*}
and similarly
\begin{align*}
|E\theta_{0,j}\theta_{k+h, \ell+h}|\leq (E\|Y_0\|^4)^{3/4}(E\|Y_0-Y_{0,k+h}\|^4)^{1/4}.
\end{align*}
Thus we get
\begin{align*}
\sum_{(j,k,\ell)\in R_{2,1}}|E\eta_{0,j}\eta_{k,\ell}E\theta_{0,j}\theta_{k+h}\theta_{\ell +h}|
\leq n
 ( E\|X_0\|^4E\|Y_0\|^4)^{3/4}J_\X(h)J_\Y(h).
\end{align*}
For $(j,k,\ell)\in R_{2,2}$ similar arguments yield
\begin{align*}
\sum_{(j,k,\ell)\in R_{2,2}}|E\eta_{0,j}\eta_{k,\ell}E\theta_{0,j}\theta_{k+h}\theta_{\ell +h}|
\leq n
 ( E\|X_0\|^4E\|Y_0\|^4)^{3/4}J_\X(h)J_\Y(h).
\end{align*}
If $(j,k,\ell)\in R_{2,3}$, then we get
\sa
|E\eta_{0,j}\eta_{k,\ell}|&=\left|\intt X_0(t)(X_k(s)X_\ell(s)X_j(t)-X_{k,k}(s)X_{\ell, k}(s)X_{j,k}(t))dtds\right|\\
&\leq 3(E\|X_0\|^4)^{3/4}(E\|X_0-X_{0,k}\|^4)^{1/4}
\end{align*}
and
\sa
|E\theta_{0,j}\theta_{k+h,\ell+h}|\leq 3(E\|Y_0\|^4)^{3/4}(E\|Y_0-Y_{0,j}\|^4)^{1/4}
\end{align*}
resulting in
\sa
\sum_{(j,k,\ell)\in R_{2,3}}&|E\eta_{0,j}\eta_{k,\ell}E\theta_{0,j}\theta_{k+h, \ell+h}|\\
&\leq 9(E\|X_0\|^4E\|Y_0\|^4)^{3/4}\sum_{(j,k,\ell)\in R_{2,3}}(E\|X_0-X_{0,k}\|^4)^{1/4}(E\|Y_0-Y_{0,j}\|^4)^{1/4}\\
&\leq 9n(E\|X_0\|^4E\|Y_0\|^4)^{3/4} J_\X(0)J_\Y(0) \frac{h}{n},
\end{align*}
since due to the restriction $j-h\leq \ell\leq j$ on $R_{2,3}$, there are only $h$ choices for $\ell$ for any $j$. Next we consider the sum of the expected values assuming that $(j,k,\ell)\in R_2$. Using \eqref{a-bau} and \eqref{b-bau} we conclude
\sa
&\sum_{(j,k,\ell)\in R_2}|a(j)a(\ell-k)b(j)b(\ell-k)|\\
&\leq E\|X_0\|^2E\|Y_0\|^2\sum_{(j,k,\ell)\in R_2}(E\|X_0-X_{0,j}\|^2 E\|X_0-X_{0,\ell-k}\|^2E\|Y_0-Y_{0,j}\|^2E\|Y_0-Y_{0,\ell-k}\|^2)^{1/2}
\end{align*}
and
\sa
&\sum_{(j,k,\ell)\in R_2}(E\|X_0-X_{0,j}\|^2 E\|X_0-X_{0,\ell-k}\|^2E\|Y_0-Y_{0,j}\|^2E\|Y_0-Y_{0,\ell-k}\|^2)^{1/2}\\
&\leq \sum_{0\leq k\leq \ell\leq N}(E\|X_0-X_{0,\ell-k}\|^2E\|Y_0-Y_{0,\ell-k}\|^2)^{1/2}\left(\sum_{j\geq \ell} (E\|X_0-X_{0,j}\|^2 E\|Y_0-Y_{0,j}\|^2)^{1/2} \right)\\
&=\sum_{\ell=0}^{ N}\left( \sum_{m=0}^\ell  (E\|X_0-X_{0,m}\|^2E\|Y_0-Y_{0,m}\|^2)^{1/2} \right)\left(\sum_{j\geq \ell} (E\|X_0-X_{0,j}\|^2 E\|Y_0-Y_{0,j}\|^2)^{1/2} \right)\\
&\leq  \sum_{m=0}^\infty  (E\|X_0-X_{0,m}\|^2E\|Y_0-Y_{0,m}\|^2)^{1/2} \sum_{\ell=0}^\infty \sum_{j=\ell} ^\infty(E\|X_0-X_{0,j}\|^2 E\|Y_0-Y_{0,j}\|^2)^{1/2}\\
&\leq 2(E\|Y_0\|^2)^{1/2}L_\X(0)  \sum_{\ell=0}^\infty \ell(E\|X_0-X_{0,\ell}\|^2E\|Y_0-Y_{0,\ell}\|^2)^{1/2}\\
%&\leq  \sum_{m=0}^\infty  (E\|X_0-X_m\|^2E\|Y_0-Y_m\|^2)^{1/2} \sum_{\ell=0}^\infty L_\X(0)L_\Y(0)\\
&\leq 4E\|Y_0\|^2L_\X(0)\sum_{\ell=0}^\infty \ell(E\|X_0-X_{0,\ell}\|^2)^{1/2}
\end{align*}
by Abel's summation formula.
Finally we consider the case when the summation is over $R_3=\{0\leq k \leq j \leq \ell\leq n-1  \}$. Let $R_{3,1}=\{(j,k,\ell)\in R_3:  0\leq k+h\leq j \leq \ell+h\}$ and $ R_{3,2}=\{(j,k,\ell)\in R_3:  0\leq j \leq k+h\leq \ell+h\}.$ If $(j,k,\ell)\in R_{3,1}$ we have
\begin{align}\label{3-1/2}
|E\eta_{0,j}\eta_{k,\ell}|&=|E\intt X_0(t)X_k(s)X_j(t)X_\ell(s)dtds|\\
&=E\intt |X_0(t)X_k(s)X_j(t)(X_\ell(s)-X_{\ell,\ell-j}(s))|dtds\notag\\
&\leq (E\|X_0\|^4)^{3/4}(E\|X_0-X_{0,\ell-j}\|^4)^{1/4}\notag
\end{align}
and similarly
\sa
|E\theta_{0,j}\theta_{k+h,\ell+h}|&=|E\intt Y_0(t)(Y_{k+h}(s)Y_j(t)Y_{\ell+h}(s)-Y_{k+h,k+h}(s)Y_{j,k+h}(t)Y_{\ell+h,k+h}(s))dtds|\\
&\leq 3 (E\|Y_0\|^4)^{3/4}(E\|Y_0-Y_{0,k+h}\|^4)^{1/4}
\end{align*}
resulting in
\sa
\frac{1}{n}\sum_{(j,k,\ell)\in R_{3,1}} |E\eta_{0,j}\eta_{k,\ell}E\theta_{k+h}\theta_{\ell+h}|\leq 3(E\|X_0\|^4E\|X_0\|^4)^{3/4}J_\X(0)J_\Y(h).
\end{align*}
In case of $(j,k,\ell)\in R_{3,2}$ the upper bound in \eqref{3-1/2} still holds but for the $\theta$'s we have
\sa
|E\theta_{0,j}\theta_{k+h,\ell+h}|&=|E\intt Y_0(t)(Y_j(t)Y_{k+h}(s)Y_{\ell+h}(s)-Y_{j,j}(t)Y_{k+h,j}(s)Y_{\ell+h,j}(s))dtds|\\
&\leq 3 (E\|Y_0\|^4)^{3/4}(E\|Y_0-Y_{0,j}\|^4)^{1/4}.
\end{align*}
Since for all $(j,k,\ell)\in R_{3,2}$ we have that $j-h\leq k \leq j$ we conclude
\sa
&\sum_{(j,k,\ell)\in R_{3,2}} |E\eta_{0,j}\eta_{k+h}E\theta_{0,j}\theta_{k+h,\ell+h}|\\
&\leq 3 (E\|X_0\|^4E\|Y_0\|^4)^{3/4}\sum_{(j,k,\ell)\in R_{3,2}}
(E\|X_0-X_{0,\ell-j}\|^4E\|Y_0-Y_{0,j}\|^4)^{1/4}\\
&\leq 3h (E\|X_0\|^4E\|Y_0\|^4)^{3/4}\sum_{(j,k,\ell)\in R_{3,2}} \sum_{0\leq j \leq \ell\leq n}(E\|X_0-X_{0,\ell-j}\|^4E\|Y_0-Y_{0,j}\|^4)^{1/4}\\
&\leq 3h (E\|X_0\|^4E\|Y_0\|^4)^{3/4}{J}_{\X}(0)J_{\Y}(0).
\end{align*}
Furthermore, by \eqref{a-bau} and \eqref{b-bau} we get
\sa
&\sum_{(j,k,\ell)\in R_{3}}|a(j)a(\ell-k)b(j)b(\ell-k)|\\
&\leq E\|X_0\|^2E\|Y_0\|^2\sum_{0\leq k\leq j \leq \ell}
(E\|X_0-X_{0,j}\|^2E\|X_0-X_{0,\ell-k}\|^2E\|Y_0-Y_{0,j}\|^2E\|Y_0-Y_{0,\ell-k}\|^2)^{1/2}\\
&\leq E\|X_0\|^2E\|Y_0\|^2\sum_{0\leq k \leq \ell}(E\|X_0-X_{0,\ell-k}\|^2E\|Y_0-Y_{0,\ell-k}\|^2)^{1/2}\\
&\hspace{2 cm}\times \sum_{j=k}^\infty(E\|X_0-X_{0,j}\|^2E\|Y_0-Y_{0,j}\|^2)^{1/2}\\
&\leq 2E\|X_0\|^2E\|Y_0\|^3 L_\X(0)\sum_{k=0}^\infty \sum_{j=k}^\infty(E\|X_0-X_{0,j}\|^2E\|Y_0-Y_{0,j}\|^2)^{1/2}\\
&\leq 4E\|X_0\|^3E\|Y_0\|^3 L_\X(0)\sum_{k=0}^\infty k (E\|Y_0-Y_{0,k}\|^2)^{1/2}.
\end{align*}
This completes the proof of the lemma.
\end{proof}
\medskip
\section{Asymptotics for means and variances of $\tilde{T}_{n,H}$ and $\bar{T}_{n,H,m}$}\label{sec-var}\setcounter{equation}{0}

We recall that $\bar{T}_{n,H,m}$ is determined by the $m$--dependent random functions $\{X_{i,m}, -\infty<i<\infty\}$ and $\{Y_{i,m}, -\infty<i<\infty\}$ defined in \eqref{xm-def} and \eqref{ym-def}, respectively.

\begin{lemma}\label{l-var} If Assumptions \ref{as-2}--\ref{as-1}  and \eqref{zero} hold, then for every $m\geq 1$ we have that
\beq\label{mean-1}
\frac{E\bar{T}_{n,H,m}}{2H+1}=\bar{\mu}(m)+O(1/n),
\eeq
$$
\bar{\mu}(m)=\sum_{\ell=-m}^m\int E[X_{0,m}(t) X_{\ell,m}(t)]dt\int E[Y_{0,m}(s)Y_{\ell,m}(s)]ds,
$$
\beq\label{mean-2}
\frac{E\tilde{T}_{n,H}}{2H+1}={\mu}+O(1/n),
\eeq
where $\mu$ is defined in \eqref{a-def}, and
\beq\label{var-1}
\lim_{n\to \infty}\frac{\mbox{{\rm var}}(\bar{T}_{n,H,m})}{2H+1}=\bar{\sigma}^2(m),
\eeq
$$
\bar{\sigma}^2(m)=\sum_{h=-m}^m \bar{\sigma}_h^2(m),
$$
$$
\bar{\sigma}^2_h(m)=2\int\!\!\!\cdots\!\!\!\int\left(\sum_{\ell=-m}^mE[X_{0,m}(t) X_{\ell,m}(s)]E[Y_{0,m}(u) Y_{\ell+h,m}(v)]\right)^2dtdsdudv,
$$
\beq\label{var-2}
\lim_{m\to \infty}\bar{\sigma}^2(m)=\sigma^2,
\eeq
where $\sigma^2$ is defined in \eqref{sigdef}.
\end{lemma}

\begin{proof} First we note that by the stationarity of $X_{i,m}$ and $Y_{j,m}$ we have
\sa
\lim_{n\to \infty}\frac{E\bar{T}_{n,H,M}}{2H+1}&=nE\int\bar{C}_{n,0,m}^2(t,s)dtds\\
&=\frac{1}{n}\sum_{1\leq i,j\leq n}\intt E[X_{i,m}(t)X_{j,m}(t)]E[Y_{i+h,m}(s)Y_{j+h,m}(s)]dtds\\
&=\int EX_{0,m}(t)X_{0,m}(t)dt\int EY_{0,m}(t)Y_{0,m}(t)ds\\
&\hspace{.5 cm}+2\sum_{1\leq \ell \leq n-1}\left(1-\frac{\ell}{n}\right) \intt E[X_{0,m}(t)X_{\ell,m}(t)]E[Y_{0,m}(s)Y_{\ell,m}(s)]dtds\\
&=\bar{\mu}(m)+O(1/n),
\end{align*}
proving   \ref{mean-1}. Similarly,
\sa
nE\tilde{\xi}_{n,h}
&=\frac{1}{n}\sum_{1\leq i,j\leq n}\intt E[X_i(t)X_j(t)]E[Y_{i+h}(s)Y_{j+h}(s)]dtds\\
&=\int EX_0(t)X_0(t)dt\int EY_0(t)Y_0(t)ds\\
&\vspace{1 cm}+2\sum_{1\leq \ell \leq n-1}\left(1-\frac{\ell}{n}\right) \intt E[X_0(t)X_\ell(t)]E[Y_0(s)Y_\ell(s)]dtds,
\end{align*}
and therefore by Assumptions \ref{as-3} and \ref{as-4} we have \eqref{mean-2}. \\
To prove \eqref{var-1} we note that by stationarity we have
\sa
\mbox{\rm var}(\bar{T}_{n,H,m})&=\sum_{h,h'=-H}^H\mbox{\rm cov}(\bar{\xi}_{n,h,m}, \bar{\xi}_{n,h',m})\\
&=(2H+1)\sum_{h=-2H}^{2H}\mbox{\rm cov}(\bar{\xi}_{n,0,m}, \bar{\xi}_{n,h,m})-\sum_{h=-2H}^{2H}|h|\mbox{\rm cov}(\bar{\xi}_{n,0,m}, \bar{\xi}_{n,h,m})\\
&=(2H+1)\sum_{h=-2H}^{2H}\mbox{\rm cov}(\bar{\xi}_{n,0,m}, \bar{\xi}_{n,h,m})+O(1),
\end{align*}
since by the arguments used in Section \ref{sec-mom}
$$
\sum_{-2H\leq h \leq 2H}|h||\mbox{\rm cov}(\bar{\xi}_{n,0,m}, \bar{\xi}_{n,h,m})|=O(1).
$$
Next we show that
\begin{align}\label{var-new}
\max_{-2H\leq h \leq 2H}\left|\mbox{\rm cov}(\bar{\xi}_{n,0,m}, \bar{\xi}_{n,h,m})-\bar{\sigma}^2_h(m)\right|=O(H/n).
\end{align}
 Using the definition of $\bar{\xi}_{n,h,m}$ we have
 \begin{align*}
 \mbox{\rm cov}(\bar{\xi}_{n,0,m}, \bar{\xi}_{n,h,m})=\frac{1}{n^2}\sum_{1\leq i,j,k,\ell\leq n}{\chi}_{i,j,k,\ell}(h),
 \end{align*}
where ${\chi}_{i,j,k,\ell}(h)={\one}_{i,j,k,\ell}(h)-{\two}_{i,j,k,\ell}(h)$,
\sa
&\one_{i,j,k,\ell}(h)\\
&=\int\!\!\!\cdots\!\!\!\int E[X_{i,m}(t)X_{j,m}(t)X_{k,m}(u)X_{\ell,m}(u)]E[Y_{i,m}(s)Y_{j,m}(s)Y_{k+h,m}(v)Y_{\ell+h,m}(v)]dtdsdudv
\end{align*}
and
\sa
&\two_{i,j,k,\ell}(h)\\
&=\int\!\!\!\cdots\!\!\!\int E[X_{i,m}(t)X_{j,m}(t)]E[Y_{i,m}(s)Y_{j,m}(s)]E[X_{i,m}(u)X_{j,m}(u)]E[Y_{i,m}(v)Y_{j,m}(v)]dtdsdudv.
\end{align*}
Let $R_1=\{ (i,j,k,\ell): 1\leq i,j\leq k,\ell\leq n\}$  and $R_2=\{(i,j,k,\ell): 1\leq i,k\leq j,\ell\leq n\}.$  Using stationarity we get
$$
\sum_{1\leq i,j,k,\ell\leq n}\chi_{i,j,k,\ell}(h)=2\!\!\!\!\!\sum_{(i,j,k,\ell)\in R_1}\chi_{i,j,k,\ell}(h)+4\!\!\!\!\!\sum_{(i,j,k,\ell)\in R_2}\chi_{i,j,k,\ell}(h).
$$
By  the $m$--dependence of $X_{i,m}$ and $Y_{i,m}$ we obtain that
\beq\label{meanR1}
\max_{|h|\leq 2H}\frac{1}{n^2}\left|\sum_{(i,j,k,\ell)\in R_1}\chi_{i,j,k,\ell}(h)\right|=O\left(\frac{1}{n}\right).
\eeq
Also, it is easy to check that
\begin{align}\label{meanR2}
\max_{|h|\leq 2H}\frac{1}{n^2}\left|\sum_{(i,j,k,\ell)\in R_2}%\int\!\!\!\cdots\!\!\!\int E[X_{i,m}(t)X_{j,m}(t)]E[Y_{i,m}(s)Y_{j,m}(s)]
%&\hspace{1 cm}\times E[X_{i,m}(u)X_{j,m}(u)]E[Y_{i,m}(v)Y_{j,m}(v)]dtdsdudv
\two_{i,j,k,\ell}(h)\right|
=O\left(\frac{1}{n}\right)%\notag
\end{align}
and therefore we need to study $\sum_{(i,j,k,\ell)\in R_2}\one_{i,j,k,\ell}(h)$ only. We decompose $R_2$ as $R_2=\cup_{i=1}^4R_{2,i}$, where
$$
R_{2,1}=\{ (i,j,k, \ell): 1\leq i \leq k \leq j\leq \ell\leq n\},\;\;
R_{2,2}=\{ (i,j,k, \ell): 1\leq i \leq k \leq \ell<j\leq n\}
$$
$$
R_{2,3}=\{ (i,j,k, \ell): 1\leq k<i \leq \ell <j\leq n\},\;\;
R_{2,4}=\{ (i,j,k, \ell): 1\leq k<i \leq j\leq \ell \leq n\}.
$$
Furthermore, let $R_{2,1}^{(1)}=\{(i,j,k,\ell): k\leq j \leq k+h+m\}\cap R_{2,1}$ and $R_{2,1}^{(2)}=\{(i,j,k,\ell): j>k+h+m\}\cap R_{2,1}$. Using again $m$--dependence we get that
$$
\frac{1}{n^2}\max_{| h|\leq 2H}\left|\sum_{(i,j,k,\ell)\in R_{2,1}^{(1)}}\one_{i,j,k,\ell}(h)\right|=O\left(\frac{H}{n}\right).
$$
Furthermore,
\sa
&\sum_{(i,j,k,\ell)\in R_{2,1}^{(2)}}\one_{i,j,k,\ell}(h)\\
&=\int\!\!\!\cdots\!\!\!\int \sum_{i=1}^n\sum_{k=i}^n\sum_{j=k+m+h}^n\sum_{\ell=j}^n E[X_{i,m}(t)X_{j,m}(t)X_{k,m}(u)X_{\ell,m}(u)]\\
&\hspace{1 cm}\times E[Y_{i,m}(s)Y_{j,m}(s)Y_{k+h,m}(v)Y_{\ell+h,m}(v)]dtdsdudv\\
&=\int\!\!\!\cdots\!\!\!\int \sum_{i=1}^n\sum_{k=i}^n\sum_{j=k+m+h}^n\sum_{\ell=j}^n E[X_{i,m}(t)X_{k,m}(u)]E[X_{j,m}(t)X_{\ell,m}(u)]\\
&\hspace{1 cm}\times E[Y_{i,m}(s)Y_{k+h,m}(v)]E[Y_{j,m}(s)Y_{\ell+h,m}(v)]dtdsdudv\\
&=\int\!\!\!\cdots\!\!\!\int\sum_{i=1}^n\sum_{r=0}^{n-i}\sum_{j=r+i+m+h}^n\sum_{p=0}^{n-j}E[X_{0,m}(t)X_{r,m}(u)]E[X_{0,m}(t)X_{p,m}(u)]\\
&\hspace{1 cm}\times E[Y_{0,m}(s)Y_{r+h,m}(v)]E[Y_{0,m}(s)Y_{p+h,m}(v)]dtdsdudv\\
&=U_{n,1}(h)+\int\!\!\!\cdots\!\!\!\int \sum_{i=1}^{n-m}\sum_{r=0}^m\sum_{j=i+m+h}^{n-m}\sum_{p=0}^mE[X_{0,m}(t)X_{r,m}(u)]E[X_{0,m}(t)X_{p,m}(u)]\\
&\hspace{1 cm}\times E[Y_{0,m}(s)Y_{r+h,m}(v)]E[Y_{0,m}(s)Y_{p+h,m}(v)]dtdsdudv
\end{align*}
and
$$
\max_{| h |\leq 2H}|U_{n,1}(h)|=O(n).
$$
Elementary algebra and $H/n\to 0$ give
$$
\sup_{|h |\leq 2H}\left|\sum_{i=1}^{n-m}\sum_{j=i+m+h}^{n-m}\!\!\! 1-\frac{n^2}{2}\right|=O(n)
$$
and therefore
\sa
&\int\!\!\!\cdots\!\!\!\int \sum_{i=1}^{n-m}\sum_{r=0}^m\sum_{j=i+m+h}^{n-m}\sum_{p=0}^mE[X_{0,m}(t)X_{r,m}(u)]E[X_{0,m}(t)X_{p,m}(u)]\\
&\hspace{1 cm}\times E[Y_{0,m}(s)Y_{r+h,m}(v)]E[Y_{0,m}(s)Y_{p+h,m}(v)]dtdsdudv\\
& =\frac{n^2}{4}\bar{\sigma}_{h,1}^2(m)+U_{n,2}(h),
\end{align*}
where
\sa
\bar{\sigma}_{h,1}^2(m)&=2\int\!\!\!\cdots\!\!\!\int \sum_{r=0}^m\sum_{p=0}^mE[X_{0,m}(t)X_{r,m}(u)]E[X_{0,m}(t)X_{p,m}(u)]\\
&\hspace{1 cm}\times E[Y_{0,m}(s)Y_{r+h,m}(v)]E[Y_{0,m}(s)Y_{p+h,m}(v)]dtdsdudv
\end{align*}
and
$$
\max_{| h| \leq 2H}|U_{n,2}(h)|=O(n).
$$
Thus we have
$$
\max_{|h| \leq 2H}\left|\frac{4}{n^2}\sum_{(i,j,k,\ell)\in R_{2,1}}\chi^{(1)}_{i,j,k,\ell}(h)-\bar{\sigma}_{h,1}^2(m)\right|=O\left(\frac{H}{n}\right).
$$
In a similar fashion it can be shown that
$$
\max_{|h| \leq 2H}\left|\frac{4}{n^2}\sum_{(i,j,k,\ell)\in R_{2,q}}\chi^{(1)}_{i,j,k,\ell}(h)-\bar{\sigma}_{h,q}^2(m)\right|=O\left(\frac{H}{n}\right),\;\;\;q=2,3,4,
$$
 where
 \sa
\bar{\sigma}_{h,2}^2(m)&=2\int\!\!\!\cdots\!\!\!\int \sum_{r=0}^m\sum_{p=-m}^{-1}E[X_{0,m}(t)X_{r,m}(u)]E[X_{0,m}(t)X_{p,m}(u)]\\
&\hspace{1 cm}\times E[Y_{0,m}(s)Y_{r+h,m}(v)]E[Y_{0,m}(s)Y_{p+h,m}(v)]dtdsdudv,
\end{align*}
 \sa
\bar{\sigma}_{h,3}^2(m)&=2\int\!\!\!\cdots\!\!\!\int \sum_{r=-m}^{-1}\sum_{p=-m}^{-1}E[X_{0,m}(t)X_{r,m}(u)]E[X_{0,m}(t)X_{p,m}(u)]\\
&\hspace{1 cm}\times E[Y_{0,m}(s)Y_{r+h,m}(v)]E[Y_{0,m}(s)Y_{p+h,m}(v)]dtdsdudv,
\end{align*}
 \sa
\bar{\sigma}_{h,4}^2(m)&=2\int\!\!\!\cdots\!\!\!\int \sum_{r=-m}^{-1}\sum_{p=0}^{m}E[X_{0,m}(t)X_{r,m}(u)]E[X_{0,m}(t)X_{p,m}(u)]\\
&\hspace{1 cm}\times E[Y_{0,m}(s)Y_{r+h,m}(v)]E[Y_{0,m}(s)Y_{p+h,m}(v)]dtdsdudv.
\end{align*}
Since $\bar{\sigma}^2_h(m)=\bar{\sigma}_{1,4}^2(m)+\cdots +\bar{\sigma}_{h,4}^2(m)$, the proof of \eqref{var-new} is complete.\\
Observing that  $\bar{\sigma}^2_h(m)=0$, if $|h|>2m+1$, and $H/n\to 0$, \eqref{var-1} is established.
\\
The series in the definitions of  $\sigma^2$ are absolutely convergent and Assumptions \ref{as-3} and \ref{as-4}  imply that $\bar{\sigma}^2_h(m)\to \sigma_h^2$, as $m\to\infty$, proving \eqref{var-2}.
\end{proof}
\end{document}